\documentclass[11pt,twoside,reqno]{amsart}

\usepackage{amssymb}
\usepackage{xcolor}
\usepackage[final]{hyperref}

\hypersetup{unicode= false, colorlinks=true, linkcolor=blue,
anchorcolor=blue, citecolor=green, filecolor=red, menucolor=blue,
pagecolor=blue, urlcolor=blue}

\def\ti{\tilde}

\def\dist{\text{\rm dist}}

\def\to{\rightarrow}

\def\ms{\medskip}

\def\R{{\mathbb R}}

\def\Z{{\mathbb{Z}}}
\def\C{{\mathbb{C}}}

\def\L{\Lambda}
\def\l{\lambda}

\def\lan{\lambda_n}

\newcommand{\he}{\mathcal{H}(E)}
\newcommand{\ta}{\tilde A}
\newcommand{\tb}{\tilde B}
\newcommand{\te}{\tilde E}

\newcommand{\vep}{\varepsilon}
\renewcommand{\phi}{\varphi}

\newcommand{\rea}{{\rm Re}\,}
\newcommand{\ima}{{\rm Im}\,}

\theoremstyle{plain}
\newtheorem{lemma}{Lemma}
\newtheorem{theorem}{Theorem}
\newtheorem{corollary}{Corollary}

\newtheorem{remark}{Remark}

\numberwithin{equation}{section}

\author{A.~Baranov}
\address{A.~Baranov 
\newline Department of Mathematics and Mechanics, St.~Petersburg State University, St.~Petersburg, Russia,
}
\email{anton.d.baranov@gmail.com}

\author{Y.~Belov}
\address{Y.~Belov \newline Chebyshev Laboratory, St.~Petersburg State University, St. Petersburg, Russia}
\email{ j\_b\_juri\_belov@mail.ru}

\author{A.~Poltoratski}
\address{A.~Poltoratski \newline Texas A\&M University
\\ Department of Mathematics\\
College Station, TX 77843, USA }
\email{alexeip@math.tamu.edu}

\thanks{The work is supported by Russian Science Foundation grant 14-41-00010.}


\title{De Branges functions of Schroedinger equations}

\begin{document}
\sloppy

\begin{abstract} We characterize the Hermite-Biehler (de Branges) functions $E$ which correspond
to Shroedinger operators with $L^2$ potential on the finite interval. From this characterization
one can  easily deduce a recent theorem by Horvath. We also obtain a result about location of resonances.
\end{abstract}

\maketitle

\ms\section{Introduction}

The Krein--de Branges theory of Hilbert spaces of entire functions was created in the mid-twentieth century to treat spectral problems for second order differential equations. The central object of the theory is a Canonical System of differential equations on the real line. The main result of the theory states that there exists a one-to-one correspondence between such systems and de Branges spaces of entire functions.
Each de Branges space is generated by a single de Branges entire function $E$  which encodes full information about the space and the
differential operator.
Via this result, spectral problems for differential operators translate 
into uniqueness and interpolation problems for spaces of entire functions. After 
such a translation, they can be viewed in a systematic way and treated 
using powerful tools of complex analysis. Since its creation, the theory 
has exceeded its original purpose and now extends to many fields of mathematics. 
Among them are complex function theory and functional model theory, spectral 
theory of Jacobi matrices and the theory of orthogonal polynomials, number 
theory and intriguing relations with the Riemann 
Hypothesis, see for instance \cite{DM, Lagarias1, Lagarias2, Romanov1,MIF}.

The Krein--de Branges theory is a classical, yet still developing area of analysis with many important open questions. Among them is a number of
'characterization problems' where a description of de Branges functions corresponding to various important sub-classes of Canonical Systems is required (see e.g., \cite[Theorems 1.4, 1.5, 6.1]{abb}). 
Through a standard procedure, see Section \ref{prelim}, many  second order 
equations and systems can be rewritten as Canonical Systems. Among them
are Schroedinger operators on an interval or half-line, Dirak systems, 
Krein strings, etc. Via the main theorem of the Krein--de Branges theory 
mentioned above, these classes of differential operators can be uniquely identified with corresponding classes of de Branges entire functions
$E$. Descriptions (characterizations) of such classes of entire functions present a whole set of interesting and challenging problems. Most of them are still open.

In this paper we solve one of such characterization problems by describing the set of de Branges functions which correspond to Schroedinger operators
with $L^2$-potentials on the interval. We restrict ourselves to the case $p=2$ due to the fact that in this case the characterization theorem, see
Theorem \ref{char}, takes an especially natural and satisfying form. 
We obtain our result from the characterization of spectra for Schroedinger operators by Chelkak \cite{Chelkak} (although, as mentioned in \cite{Chelkak}, such a spectral characterization in the case $p=2$ was  known earlier).
The difference with $p\neq 2$ is largely due to the fact that the Fourier transform of an $L^p$-function is not as easy to describe as in the $p=2$ case. Nonetheless, our methods allow us to study characterization theorems
for general $L^p$ and other classes of potentials. Such studies will be presented elsewhere.

To illustrate applications of characterization theorems for de Branges functions, 
in Section \ref{horv} we obtain a recent theorem by M. Horvath \cite{Horvath} which 
is considered by many experts to be the
strongest result in that area. As was demonstrated in \cite{Horvath}, 
it implies a number of classical and recent results
such as Ambarzumian's theorem, Borg's two-spectra theorem, 
the results by Hoschtadt--Lieberman, Gesztesy--Simon, del Rio--Simon and several others.
Achieving an understanding of connections between Horvath's 
methods with the Krein--de Branges theory served as initial motivation for this paper.

In addition in Section \ref{res} we prove that some logarithmic strip 
is free of zeros of $E$.  For physists zeros of $E$ are known as {\it resonances} 
(poles of the scattering matrix), see \cite{simon, korot}. 
It is easy that this result is essentially 
sharp and we give a corresponding example.

\textbf{Organization of the paper.} The paper is organized as follows.
In Section \ref{prelim} we remind the reader the basics of the theory and give further references.
In Section \ref{mainth} we formulate and prove the main result of the paper. 
Sections \ref{horv} and \ref{res}
are devoted to applications.


\ms\section{Preliminaries}
\label{prelim}

\subsection{De Branges spaces} Consider an entire function  $E(z)$ satisfying
the inequality
$$ 
|E(z)|>|E(\bar{z})|, \qquad  z \in \C_{+},
$$ 
and such that $E\ne 0$ on $\mathbb{R}$. 
Such functions are usually called  {\it de~Branges functions} (or 
{\it Hermite--Biehler functions}). The {\it de~Branges 
space} $\he$ associated with $E$ is defined to be the space of
entire functions $F$ satisfying
$$ 
\frac{F(z)}{E(z)} \in H^2(\C_+), \hspace{1cm}
\frac{F^{\#}(z)}{E(z)} \in H^2(\C_+),
$$ 
where $F^\#(z) = \overline{F(\overline{z})}$, $H^2(\C_+)$ is a Hardy class in $\C_+$.
It is a Hilbert space
equipped with the norm $\|F\|_E =\|F/E\|_{L^2(\R)}.$ 
If $E(z)$ is of exponential type then
all the functions in the de~Branges space $\he$ are of
exponential type not greater then the type of $E$  
(see, for example, the last part in the proof of Lemma 3.5 in~\cite{DM}). 
A de~Branges space is called {\it short} (or {\it regular}) if together with
every function $F(z)$ it contains $(F(z)-F(a))/(z-a)$ for any
$a\in\C$.

One of the most important features of de Branges spaces 
is that they admit a second, axiomatic, definition. Let $\mathcal{H}$ be
a Hilbert space of entire functions that satisfies the following axioms:

\begin{itemize}

\item (A1) For any $\l\in\C$, point evaluation at $\l$ is a non-zero bounded 
linear functional on $\mathcal{H}$;

\item (A2) If $F\in \mathcal{H},\ F(\l)=0$, then 
$F(z)\frac{z-\bar\l}{z-\l}\in \mathcal{H}$ with the same norm;

\item (A3) If $F\in \mathcal{H}$ then $F^\#\in \mathcal{H}$ with the same norm.

\end{itemize}

Then $\mathcal{H}=\he$ for a suitable de Branges function $E$. 
This is Theorem 23 in \cite{dB}.


\subsection{Canonical Systems}
\label{canon}
Let $\Omega$ be a symplectic matrix in $\R^2$:

$$
   \Omega=\left(\begin{array}{cc}
    0 & 1 \\
   -1 & 0
   \end{array} \right).
$$

A \textit{canonical system} of differential equations (CS) is the system
\begin{equation}
\Omega \dot X(t)=zH(t)X(t),\label{e1}
\end{equation}
where $H$ is a real, locally summable, $2\times 2$-matrix 
valued function on an interval $(a,b)\subset \mathbb{R}$, 
called a Hamiltonian of the system, and
$X$ is an unknown vector valued function $X=\left(\begin{array}{c}
    u \\
   v
   \end{array} \right)$. Such systems were considered by M.~Krein as a general form of a second order linear differential operator. As was mentioned in the introduction, many standard classes of second order equations can be equivalently rewritten as canonical systems.

  Canonical systems and de Branges spaces together constitute the so-called Krein--de Branges theory. The connection between the two is as follows.
  Let $X=\left(\begin{array}{c}
    A(z,t) \\
   B(z,t)
   \end{array} \right)$ be a solution to \eqref{e1} 
   satisfying some initial condition at $a$, for instance $X(a)=
   \left(\begin{array}{c}
    0\\
    1
   \end{array} \right)$. Then for any fixed $t\in (a,b)$, 
the function $E_t(z)=A(z,t) + i B(z,t)$ is a de Branges function.
Under some minor technical restrictions 
on the Hamiltonian, the corresponding spaces $\mathcal{H}(E_t)$ are nested, i.e.,
$\mathcal{H}(E_t)$ is isometrically embedded
into $\mathcal{H}(E_s)$ for any $t<s$. In the opposite direction, 
any de Branges function $E$ can be obtained 
this way from a canonical system \eqref{e1}, see Theorem 40 in \cite{dB} or Theorems 16-18
\cite{Romanov2}. 
The solution $X$ can be used as a kernel of an integral operator (Weyl transform)
which identifies the space of vector-valued square-summable functions on 
$(a,t)$ with $\mathcal{H}(E_t)$. For more on Krein--de Branges theory see 
\cite{DM, dB, Romanov2}.


\subsection{Schroedinger operators} 
\label{schr}
This paper is devoted 
to a particular case of a canonical system, the Schroedinger 
equation on an interval $(a,b)$,
\begin{equation}
\label{e2}
-\ddot u +qu= z u.
\end{equation}

Since this equation can be rewritten as a canonical system (see for instance \cite{MIF, Romanov2}), it corresponds to a chain of de Branges functions/spaces as described above. Let us present a shorter way to establish this connection, without transforming the equation to the canonical form
\eqref{e1}, e.g. \cite{MIF}.

Assume that the potential $q(t)$ is  integrable on a finite interval $(a,b)$, i.e. that the operator we consider is regular.
We make this restriction for the sake of simplicity, although the theory can be extended to general non-regular cases, see \cite{MIF}.
Let $u_z(t)$ be the solution of \eqref{e1}  satisfying the boundary conditions $u_z(a)=0$ and $\dot u_z(a)=1$.
The Weyl $m$-function is defined as
$$\textbf{m}(z)=\frac{u_z(b)}{\dot u_z(b)}.$$
It is well known that the Weyl function is a Herglotz function, i.e., 
a meromorphic function in $\C$ with a positive imaginary part in $\C_+$.

Now let us assume that the operator with Neumann boundary conditions 
at $a$ and Dirichlet boundary conditions at $b$ is positive (otherwise one can add a large positive constant to $q$).
In this case one may consider the Weyl function after the 'square root transform', i.e. after a change of variables
$m(z)=z\textbf{m}(z^2)$. If the operator is positive then the modified Weyl function $m(z)$ is also a Herglotz function.

The entire function
\begin{equation}
  E(z)=zu_{z^2}(b)+i\dot u_{z^2}(b) 
\label{Eform}
\end{equation}
is the de Branges function for the corresponding canonical system. 
In particular, if $q\equiv 0$, then $E(z)=ie^{-i z}$, and we get the 
classical Paley--Wiener space as the corresponding de Branges space $\he$.

Closely related analytic (meromorphic) function is the Weyl inner function
$$
\Theta=\frac{m-i}{m+i}\ \ \ or \ \ \ \Theta=\frac{E^\#}E,
$$
where $E^\#(z)=\overline{E(\bar z)}$.
Such functions will also be used in our discussions below. Each of the 
functions, $E(z)$, $m(z)$, $\textbf{m}(z)$ or $\Theta(z)$, determine the 
operator uniquely, as follows from classical results of Marchenko \cite{march1, march2}.


\ms\section{A characterization theorem}
\label{mainth}

\subsection{Spectra of the Schroedinger operators} 
Every de Branges function comes from a canonical system, 
as follows from the main result of Krein--de Branges theory. 
But which of the de Branges functions can be obtained from 
Schroedinger operators via the procedure described above? 
How the properties of the potential translate into the properties of $E$?

Once again, let us consider the Schroedinger equation

\begin{equation}-\ddot u +qu= z u\label{e3}\end{equation}
on the interval $[0,1]$.
Denote by $\sigma_{DD}$ the spectrum of $L$ with Dirichlet 
boundary conditions at the endpoints after the square root transform, i.e.,
$\sigma_{DD}=\{\lan^2\}$ such that for each $\lan$ there 
exists a solution of \eqref{e3} with $z=\lan^2$ satisfying
$u(0)=u(1)=0$. Similarly, $\sigma_{ND}$ will denote the spectrum for the mixed 
Neumann/Dirichlet conditions,
$\dot u(0)= u(1)=0$. We assume that $q\in L^1([0,1])$ 
is such that both operators with DD and ND boundary conditions 
are positive (otherwise add a positive constant to $q$). 
In that case both spectra are real. 

As defined in Subsections \ref{canon} and \ref{schr}, $E(z)$ and $m(z)$ will 
denote the de Branges and Weyl functions of the Schroedinger operator.
Another standard object associated with $E(z)$ is the phase function 
$\phi(x)=-\arg E(x)$, a continuous branch of the argument of $E$ on the
real line.

Note, that in terms of $\phi$ the spectra of the operator can be identified as
$$
\pm\sqrt{\sigma_{DD}}\cup\{0\}=\{x:\,\phi(x)=n\pi\},\qquad \pm\sqrt{\sigma_{ND}}=\biggl{\{}x:\,\phi(x)=n\pi +\frac \pi 2\biggr{\}},
$$

Since the phase function is always a growing function on $\R$, we see that the spectra 
$\sigma_{DD},\ \sigma_{ND}$ are alternating sequences,
which is a well-known fact of spectral theory. If $E$ is a de Branges function, it can be represented
as $E(z)= A(z)+iB(z)$ where $A$ and $B$ are entire functions which are real on the real line. These functions have alternating zero sequences
on $\R$ and can be viewed as analogs of sine and cosine. 
In terms of these functions, the spectra are seen as the zero sets
(note that by the constructions $A$ is odd and $B$ is even see \eqref{Eform}):
$$
\pm\sqrt{\sigma_{DD}}\cup\{0\}=\{x:\,A(x)=0\},\qquad \pm\sqrt{\sigma_{ND}}=\{x:\,B(x)=0\}.
$$

To prove our main results we utilize the following 
characterization of the spectra of regular Schroedinger 
operators by D. Chelkak \cite{Chelkak}.

\begin{theorem} 
\label{chel}
Two alternating sequences $\{\lan^2\}$ and $\{\mu^2_n\}$ on $\R$ are equal to the spectra, $\sigma_{DD}$ and $ \sigma_{ND}$ correspondingly,
for some Schroedinger operator on the interval $[0,1]$ with $q\in L^2([0,1])$
if and only if they satisfy the asymptotics
\begin{equation}
\label{aslm}
\lan^2=\pi^2 n^2+C+a_n, \qquad \mu^2_n=\pi^2\biggl{(}n-\frac{1}{2}\biggr{)}^2+C+b_n,\qquad n\in\mathbb{N}, 
\end{equation}
for some real $C$ and some $\{a_n\},\{b_n\}\in l^2$.
\end{theorem}


To obtain this statement from the main result of \cite{Chelkak} 
one needs to note that DD spectrum for even potentials $q$ (i.e., $q(t)\equiv q(1-t))$ is an union
of DD and ND spectra for the interval $[0,1\slash2]$.

Applying the square root transform we get
\begin{equation}
\lambda_n=\pi n+\frac{C}{n}+\frac{a_n}{n},\qquad \mu_n = \pi n - \frac{\pi}{2} + \frac{C}{n}+\frac{b_n}{n},\qquad n\geq 1,
\label{as}
\end{equation}
for some real $C$ and $\{a_n\}$, $\{b_n\}\in\ell^2$. Put 
$\lambda_{-n}=-\lambda_n$, $\mu_{-n}=-\mu_n$, $n\geq1$ and $\lambda_0=0$.

\begin{remark}
{\rm The constant $C$ in Theorem \ref{chel} is equal 
to $\int_0^1 q(x)dx$, see \cite{Chelkak}. }
\end{remark}

In what follows we will use essentially the fact that the sequences $\{\lambda_n\}$
and $\{\mu_n\}$ are complete interpolating sequences in the {\it Paley--Wiener space}
$PW_1$, the space of all entire functions of exponential type at most 1 which are
in $L^2(\mathbb{R})$ or, equivalently, the Fourier image of the space $L^2(-1,1)$.
Recall that $\{\lambda_n\}\subset\mathbb{R}$ is said to be  a 
{\it complete interpolating sequence} for $PW_1$ if for any sequence $\{w_n\}\in \ell^2$
there exists a unique function $f\in PW_1$ such that $f(\lambda_n) = w_n$.
For a discussion of the interpolation theory and related Riesz bases of reproducing 
kernels in the Paley--Wiener spaces we refer to \cite{hnp, seip}. The fact that
$\{\lambda_n\}$ and $\{\mu_n\}$ are complete interpolating sequences follows
from the basic results of the theory (e.g., Kadets 1/4 theorem).


\subsection{Main theorem}
The main result of this paper is the following characterization theorem. 
We say that a de Branges entire function $E(z)$ corresponds
to a Schroedinger equation \eqref{e1} if it can be obtained 
from it following the procedure described in Subsection \ref{schr}.
To avoid inessential technicalities from now on we will assume that $q$
with DD and ND boundary conditions generates positive operators.

\begin{theorem}
\label{char}
Let $E=A+iB$ be a de Branges entire function of exponential type 1
and be of Cartwright class. 

1. $E$ corresponds to a Schroedinger equation on $[0,1]$ with $q\in L^2([0,1])$ 
if and only if 
\begin{equation}
\label{char1}
z(A(z)\cos z - B(z)\sin z)= f(z)+C
\end{equation}
for some real constant $C$ and some even real-valued function $f\in L^2(\R)$.

2. Assume that $\tilde E = \tilde A+i\tilde B$
corresponds to a Schroedinger equation on $[0,1]$ with $q\in L^2([0,1])$.
Then $E=A+iB$ corresponds to a Schroedinger equation on $[0,1]$ with 
a square-integrable potential if and only if 
\begin{equation}
\label{char2}
z(A(z)\tb(z) - \ta(z)B(z))= f(z)+C
\end{equation} 
for some real constant $C$ and some even real-valued function $f\in L^2(\R)$.

3. There exists $\vep>0$ such that for any even function $f\in PW_2$ which is real 
on $\mathbb{R}$ with $\|f\|_2 <\vep$ there exists a de Branges function $E=A+iB$ which 
corresponds to a Schroedinger equation on $[0,1]$ with $q\in L^2([0,1])$ 
such that 
\begin{equation}
\label{char3}
z(A(z)\cos z - B(z)\sin z)= f(z)-f(0).
\end{equation} 
\end{theorem}

\begin{remark}
{\rm It is well-known that the functions $A$ and $B$ in the statement 
are entire functions of exponential type 1. 
Hence the function $f$ in statements 1 and 2 
is actually a function from the Paley--Wiener space $PW_2$. }
\end{remark}

\begin{remark}
{\rm The property \eqref{char2} may be rewritten as $\rea(zE\tilde E) 
\in Const + L^2(\mathbb{R})$, where we denote by $Const$ the class 
of constant functions. }
\end{remark}


\subsection{Preliminary estimates} 
\label{prelimest}

In the proof of Theorem \ref{char} we will use the following simple lemma 
about canonical products with zeros satisfying the asymptotics \eqref{as}.
    
\begin{lemma}
\label{entire} 
Let $A$ and $B$ be  Cartwright class entire functions which are real on $\mathbb{R}$
and whose zeros $\lambda_n$ and $\mu_n$
satisfy the asymptotics \eqref{as}.
Then
\begin{equation}
 \label{sin}
 \bigg|\frac{A(z)}{\sin  z}\bigg| \asymp 
 \frac{\dist(z, \{\lambda_n\})}{\dist \big(z, \Z \big)}, 
 \qquad
 \bigg| \frac{B(z)}{\cos  z} \bigg| 
 \asymp \frac{\dist(z, \{\mu_n\})}{\dist \big(z, \Z+\frac{1}{2}\big)},
\end{equation}
and, moreover,
\begin{equation}
 \label{sin1}
 A(\pi n) = (-1)^n C_1 \Big(\frac{C}{n} +\frac{\alpha_n}{n} \Big), \quad
 B\Big(\pi n+\frac{\pi}{2} \Big) = (-1)^{n+1}  C_2 \Big(\frac{C}{n} +\frac{\beta_n}{n} 
 \Big),
\end{equation}
where $C_1, C_2$ are real nonzero constants and $\{\alpha_n\}, \{\beta_n\} \in \ell^2$.
\end{lemma}

\begin{proof} 
We prove the formulas \eqref{sin}--\eqref{sin1} for the function $A$. The case 
of the function $B$ is similar.
Since $A$ is a Cartwright class functions real on $\R$, 
it may be represented as principal value products
(with obvious modification if $0\in \{\lambda_n\}$):
\begin{equation}
\label{cart}
A(z) = K z\lim_{R\to \infty} \prod_{0<|\lambda_n|\le R} \Big(1- \frac{z}{\lambda_n}\Big),
\end{equation}
where $K\in \mathbb{R}$. The formula for $B$ is analogous. 

Let $|z|$ be sufficiently large and let $n = n(z)$ be the closest integer to $z$.
Then we have 
\begin{equation}
\label{sl}
\frac{A(z)}{\sin z} = K 
\frac{z-\lambda_n}{z-\pi n}
\cdot \prod_{k\ne 0} \frac{\pi k}{\lambda_k}\cdot
\prod_{k\ne n} \bigg(1 - \frac{\lambda_k - \pi k}{z-\pi k}\bigg).
\end{equation}
Clearly, 
$$
\sum_{k\ne n(z)} \bigg|\frac{\lambda_k - \pi k}{z-\pi k}\bigg| \to 0, 
\quad |z|\to \infty,
$$
which implies the first estimate in \eqref{sin1}
and even a stronger asymptotics 
\begin{equation}
\label{sin2}
\frac{A(z)}{\sin  z} \sim C_1 \frac{z-\lambda_{n(z)}}{z- \pi n(z)}, \quad |z|\to \infty, 
\qquad  
C_1 =  K \prod_{k\ne 0} \frac{\pi k}{\lambda_k}.
\end{equation}
Moreover, it is easy to see that 
\begin{equation}
\label{sl1}
\bigg\{ \sum_{k\ne n} \log \bigg(1 - \frac{\lambda_k - \pi k}{z_n -\pi k}\bigg)\bigg\} 
\in \ell^2
\end{equation}
as a sequence enumerated by $n$ 
for any sequence $z_n$ such that $z_n \notin \pi\mathbb{Z}$ and 
$\{z_n-\pi n\}$ is bounded 
(one can refer to the boundedness in $\ell^2$
of the discrete Hilbert transform or estimate the sum directly). 
Thus, when we put $z=\pi n$ in \eqref{sl}, 
we obtain equality \eqref{sin1} for $A$ 
with the same constant $C_1$ as in \eqref{sin2}.
\end{proof}

Let us list several further corollaries of \eqref{sl}.

\begin{corollary} 
In the conditions of Lemma \ref{entire} we have
\begin{equation}
\label{est1}
\{A(\mu_n) - C_1(-1)^n\} \in \ell^2, \qquad 
\{B(\lambda_n) - C_2(-1)^n\} \in \ell^2;
\end{equation}
\begin{equation}
\label{est2}
\{A'(\lambda_n) - C_1(-1)^n\} \in \ell^2, \qquad 
\{B'(\mu_n) - C_2(-1)^{n+1}\} \in \ell^2;
\end{equation}
\begin{equation}
\label{est3}
\frac{A(iy)}{C_1 \sin iy} - 1= \frac{2 C}{y} + o\Big(\frac{1}{|y|}\Big),
\qquad
\frac{B(iy)}{C_2 \cos iy} - 1= \frac{2 C}{y} + o\Big(\frac{1}{|y|}\Big), \qquad
|y|\to\infty.
\end{equation}
\label{cor1}
\end{corollary}

\begin{proof} 
Inclusions \eqref{est1}--\eqref{est2} follow in a straightforward way from 
\eqref{sl} and the inclusion \eqref{sl1}.
We omit the details.

Now we prove \eqref{est3}. We have

\begin{equation}
\label{A1longest}
\begin{aligned}
\frac{A(z)}{C_1\sin z} & = 
\exp\biggl{[}\sum_{k\neq0}\log\biggl{(}1-\frac{\lambda_k-\pi k}{z-\pi k}\biggr{)}
\biggr{]} \\
& = \exp\biggl{[}\sum_{k\neq0}\frac{\pi k-\lambda_k}{z-\pi k}+o\biggl{(}\frac{1}{\dist(z,\pi\mathbb{Z})^2}\biggr{)}\biggr{]},
\end{aligned}
\end{equation}

when $|z|\rightarrow\infty$.
Taking $z=iy$, we get 
$$
\begin{aligned}
\frac{A(iy)}{C_1\sin iy} & = 
1+\sum_{k\neq0}\frac{\pi k-\lambda_k}{iy-\pi k}+o\biggl{(}\frac{1}{|y|}\biggr{)}
=1+\sum_{k=1}^\infty\frac{2\pi k(\lambda_k-\pi k)}{\pi^2k^2+y^2}+o\biggl{(}\frac{1}{|y|}\biggr{)} \\
&=1+2\pi\sum_{k=1}^\infty\frac{C+a_k}{\pi^2k^2+y^2}+o\biggl{(}\frac{1}{|y|}\biggr{)}=1+\frac{2 C}{y}+o\biggl{(}\frac{1}{|y|}\biggr{)}
\end{aligned}
$$
(here we used the fact that $\sum_{k=1}^\infty\frac{|y|}{\pi^2k^2+y^2}\rightarrow\pi^{-1}$ as $|y|\rightarrow\infty$).

The proof for $B$ is analogous.
\end{proof}
\medskip


\subsection{Proof of Theorem \ref{char}} 
\label{proofchar}

In this subsection we prove Theorem \ref{char}. 
\medskip
\\
{\bf Necessity of \eqref{char1} and \eqref{char2}.} 
Assume that $E=A+iB$
corresponds to a Schroedinger equation on $[0,1]$ with $q\in L^2([0,1])$.
Then, by Theorem \ref{chel}, the zeros $\lambda_n$ of $A$ and $\mu_n$ of $B$ 
satisfy the asymptotics \eqref{as}. By Lemma \ref{entire} (see \eqref{sin2}),
we have
$$
\frac{A(iy)}{\sin  iy} \sim C_1, \qquad  
\frac{B(iy)}{\cos  iy} \sim C_2,\qquad 
y\to \infty. 
$$
Since $E$ corresponds to a Schroedinger equation, we have $A(iy)/B(iy) \sim 
\sin iy/\cos iy$ (see \cite[Theorem 2.2.1]{march2}) and so we conclude that $C_1=C_2$.

Put $F(z) = z(A(z)\cos z - B(z)\sin z)$. By \eqref{sin1}, we have
$$
 F(\pi n) = \pi n\, A(\pi n) \cos \pi n = \pi C_1 C + \tilde \alpha_n, \qquad
 F\big(\pi n + \frac{\pi}{2}\big) = \pi C_1 C +\tilde \beta_n,
$$ 
for some $\{\tilde \alpha_n\}, \{\tilde \beta_n\}\in\ell^2$.
Hence, $f=F-\pi C_1 C$ is an entire function of exponential type at most $2$
such that $\big\{f\big( \frac{\pi m}{2} \big)\big\}_{m\in\mathbb{Z}} \in \ell^2$ and $f$ is real on the real line and even.
Since $\big\{\frac{\pi m}{2}\big\}_{m\in\mathbb{Z}}$ is a complete interpolating 
sequence for $PW_2$, there exists a unique function $g\in PW_2$ such that
$f\big( \frac{\pi m}{2} \big) = g\big( \frac{\pi m}{2} \big)$. 
Therefore $f-g$ vanishes on $\big\{\frac{\pi m}{2}\big\}_{m\in\mathbb{Z}}$
and so $f(z)-g(z) = h(z)\sin 2\pi z$ for some entire function $h$. Since $A,B$ 
are Cartwright class entire functions of type 1, we conclude that $h$ is 
of zero exponential type. 

Let us show that $h\equiv 0$. Indeed, 
$$
h(z) = z\bigg(\frac{A(z)}{2\sin z} - \frac{B(z)}{2\cos z}\bigg) - \frac{g(z)+C}{\sin 2z}.
$$
It follows from \eqref{est3} that
$$
y\bigg(\frac{A(iy)}{\sin iy} - \frac{B(iy)}{\cos iy}\bigg) = o(1), \qquad |y| \to \infty, 
$$
Since also $|g(iy)/\sin(2iy)| \to 0$ we conclude that $|h(iy)| \to 0$, $|y|\to \infty$, 
and finally $h\equiv 0$ by the standard Phragm\'en--Lindel\"of principle.

The proof of necessity of \eqref{char2} in statement 2 is analogous and we omit it.
\medskip
\\
{\bf Sufficiency
of \eqref{char1} and \eqref{char2}.} 
We will prove the more general statement about 
sufficiency of \eqref{char2}.
Assume that $\te$ corresponds to a Schroedinger equation
and let $A$ and $B$ satisfy \eqref{char2} for some $f\in PW_2$ and $C\in\mathbb{R}$.
It remains to show that the zeros of $A$, $B$ satisfy \eqref{as}. 

Let us denote the zeros of $\ta$ and $\tb$ by $t_n$ and $s_n$ respectively.
Then, comparing the values at $t_n$ and $s_n$ we get
$$
  A(t_n) = \frac{1}{\tb(t_n)}\bigg(\frac{C}{t_n} + \frac{f(t_n)}{t_n}\bigg), 
  \qquad 
  B(s_n) = -\frac{1}{\ta(s_n)} \bigg(\frac{C}{s_n} + \frac{f(s_n)}{s_n}\bigg),\qquad n\neq0.
$$

By \eqref{sin} we see that $|\tb(t_n)| \asymp |\ta(s_n)| \asymp 1$.
Since both $t_n$ and $s_n$ are complete interpolating sequences for $PW_1$,
there exist unique functions $g, h\in PW_1$ 
such that $g(0)=0$ and
$$
g(t_n) = \frac{1}{\tb(t_n)}\bigg(\frac{C}{t_n} + \frac{f(t_n)}{t_n}\bigg),
\qquad
h(s_n) = -\frac{1}{\ta(s_n)} \bigg(\frac{C}{s_n} + \frac{f(s_n)}{s_n}\bigg), \qquad n\neq0.
$$
Since $\tilde{A}$ is odd and $\tilde{B}$ is even we have $g(t_{-n})=g(-t_n)=-g(t_n)$ and $h(s_{-n})=h(s_n)$.
Hence, $g,h$ are real on $\mathbb{R}$, $g$ is odd and $h$ is even.
The function $A-g$ vanishes on $\{t_n\}$ and we conclude that $A-g = \ta h$ 
for some entire function $\tilde g$. Both $A, \ta$ are Cartwright class entire functions 
of type 1 and so $\tilde g$ is a constant. Thus, $A= g+\beta\ta$ and, analogously, 
$B = h+\beta \tb$ (equation \eqref{char2} implies that the constant is the same).
Without loss of generality assume in what follows that $\beta=1$. Note that $A$ is odd and $B$ is even.

Let us study the zero asymptotics for $A = \ta+g$, the case of $B = \tb+h$ is analogous.
Recall that $\{t_n\}$ is a complete interpolating sequence 
for $PW_1$ and so the functions $\frac{\ta(z)}{\ta'(t_n)(z-t_n)}$ 
form a Riesz basis in $PW_1$. Hence, expanding $g$ with respect to this Riesz basis, 
we get
$$
g(z) = \sum_{n\in\mathbb{Z}} \frac{\ta(z)}{\ta'(t_n) \tb(t_n) (z-t_n)}\cdot
\bigg(\frac{C}{t_n} + \frac{f(t_n)}{t_n}\bigg).
$$
Consider the equation $\ta+g_0 = 0$ or, equivalently, 
\begin{equation}
\label{sl4}
1 + \sum_{n\in\mathbb{Z}} \frac{1}{\ta'(t_n) \tb(t_n) (x-t_n)}\cdot 
\bigg(\frac{C}{t_n} + \frac{f(t_n)}{t_n}\bigg) = 0.
\end{equation}
It is easy to see the sum in \eqref{sl4} tends to zero as $x\to\infty$
and $x$ is separated from $\{t_n\}$. Thus, 
for any $\delta>0$ the expression in the left-hand side of \eqref{sl4}
is positive on $(t_n+\delta, t_{n+1}-\delta)$ 
for sufficiently large $|t_n|$, and it has a unique zero in the interval 
$(t_n-\delta, t_n+\delta)$. This zero (denote it by $\lambda_n$)
satisfies 
$$
\lambda_n = t_n - \frac{1}{\ta'(t_n) \tb(t_n)}\cdot
\bigg(\frac{C}{t_n} + \frac{f(t_n)}{t_n}\bigg) \cdot \bigg(
1 + \sum_{k\ne n} \frac{1}{\ta'(t_k) \tb(t_k) (\lambda_n-t_k)}
\cdot\bigg(\frac{C}{t_k} + \frac{f(t_k)}{t_k}\bigg) \bigg)^{-1}.
$$
Again, it is easy to see that 
$$
\bigg\{\sum_{k\ne n} \frac{1}{\ta'(t_k) \tb(t_k) (\lambda_n-t_k)}
\cdot\bigg(\frac{C}{t_k} + \frac{f(t_k)}{t_k}\bigg)\bigg\}_{n\in\mathbb{Z}} \in \ell^2
$$
for any choice of $\lambda_n \notin \{t_k\}$ such that
$\lambda_n \in (n-1/3, n+1/3)$ for sufficiently large $|n|$.
Combining this with the asymptotics \eqref{est1}--\eqref{est2} of $\tb(t_n)$
and $\ta'(t_n)$ we conclude that 
\begin{equation}
\label{as1}
\lambda_n = t_n - \frac{C}{t_n} +\frac{\alpha_n}{t_n} = \pi n +\frac{\tilde C}{n} 
+\frac{\tilde \alpha_n}{n},
\end{equation} 
where $\tilde C$ is some constant and $\{\alpha_n\}, 
\{\tilde\alpha_n\} \in \ell^2$.
Thus, we have shown that the zeros $\lambda_n$ 
of $\ta +g$ have the required asymptotics. 
%
%
\medskip
\\
{\bf Proof of Statement 3 of Theorem \ref{char}.}
This statement essentially is already proved. 
Let $f\in PW_2$ be given. We need to find the functions $A$, $B$ such that
\eqref{char3} is satisfied. Comparing the values at $\pi n$, 
and $\pi n+\frac{1}{2}$ we get
$$
A(\pi n)  = (-1)^n \frac{f(\pi n) - f(0)}{\pi n}, \quad n\ne 0,
\qquad A(0) = f'(0),
$$
$$
B\Big(\pi n + \frac{\pi}{2}\Big) = (-1)^{n+1} 
\frac{f(\pi n +\pi/2) - f(0)}{\pi n +\pi/2}.
$$
There exist unique functions $g, h\in PW_1$ (which are real on $\R$) 
such that 
$$
g(0)= f'(0), \qquad 
g(\pi n) = (-1)^n \frac{f(\pi n) - f(0)}{\pi n},\quad n\ne 0,
$$
$$
h\Big(\pi n + \frac{\pi}{2}\Big) = (-1)^{n+1} 
\frac{f(\pi n +\pi/2) - f(0)}{\pi n +\pi/2}.
$$
Now put $A = \sin z + g$, $B = \cos z+h$. It is clear that the function 
$$
Q(z)  = z(A(z)\cos z - B(z)\sin z) = z(g(z) \cos z- h(z) \sin z) 
$$
coincides with $f-f(0)$ at the points $\big\{\frac{\pi m}{2}\big\}_{m\in\mathbb{Z}}$.
Hence, $Q(z) = f(z)-f(0) + P(z) \sin 2z$ for some entire function $P$.
Since $f\in PW_2$ and $Q\in zPW_2$, a standard Phragm\'en--Lindel\"of
principle shows that $P$ is at most constant. Since additionally 
$Q'(0) = f'(0)$ we conclude that $P\equiv 0$.

Thus, the constructed functions $A$ and $B$ satisfy equality \eqref{char3}. Now arguing 
as in the proof of sufficiency of \eqref{char2} we may conclude that
the zeros $\lambda_n$ and $\mu_n$ of $A$ and $B$ have asymptotics \eqref{as}.
Moreover, $A$ and $B$ will be Cartwright class functions given by the
infinite products of the form \eqref{cart}.

The only difference with the above argument is that in statements 1 and 2 
we assumed from the very beginning that $E=A+iB$ is a de Branges function, 
while in statement 3 we must
prove that $E$ is a de Branges function. It is well known that if $A$ 
and $B$ are zero genus canonical products of the form \eqref{cart} (with the constants 
$K>0$) then $E=A+iB$ is a de Branges function 
if and only if the zeros of $A$ and $B$ are interlacing. However, in general 
we can only conclude from the representations \eqref{char1}--\eqref{char2}
that the zeros of $A$ and $B$ are interlacing with possible 
exception of a finite number of zeros. 

To prove the interlacing property 
we need to use the fact that the norm of $f$ is sufficiently small. 
Indeed, the norms of functions $g$ and $h$ depend on the norm of $f$ 
and if $\|f\|_2$ is sufficiently small, then it is easily seen from the previous 
arguments that both $C=-f(0)$ and $\tilde \alpha_n$ in \eqref{as1}
are sufficiently small and we can get $|\lambda_n - \pi n|< 1/2$ 
{\it for all} $n\in\mathbb{Z}$. Analogously, we get $|\mu_n - \pi n-\pi/2|<1/2$ 
when $\|f\|_2$ is sufficiently small, whence $\lambda_n$ and $\mu_n$ are interlacing.
\qed


\ms\section{Horvath' theorem}
\label{horv}

In this section we illustrate applications of our results by obtaining a recent theorem by M. Horvath \cite{Horvath}.

If $\L$ is a sequence of real points, we denote by $\sqrt\L$ the set $\{z|z^2\in\L\}$. The notation
$\sqrt\L\cup\{*,*\}$ stands for the set obtained from $\sqrt\L$ by addition of any two real numbers (not from $\L$).

We say that $\L\in\R$ is a  defining set in the class $Schr(L^2,D)$ of Schroedinger operators on $[0,1]$ with $L^2$-potential and Dirichlet boundary condition at $0$ if there do not exist two different operators $L,\ti L$ from this class whose Weyl functions $\textbf{m}$ and $\ti {\textbf{m}}$ are equal on $\L$.

A version of the following theorem is proved in \cite{Horvath} for all $1\leq p\leq \infty$. In this paper we treat only the case $p=2$, although a similar argument can
be applied to other $p$.

\begin{theorem} 
A set $\L\in\R$ is a  defining set in the class $Schr(L^2,D)$ 
if and only if $\sqrt\L\cup\{*,*\}$ is a uniqueness set
in the Paley--Wiener space $PW_2$.
\end{theorem}

\begin{proof} A simple proof of the 'if' part was given in \cite{MIF}. 
Here we present a version of it for reader's convenience. The 'only if' part
follows from Theorem \ref{char}.
\medskip
\\
\textbf{If:} Suppose that $\textbf{m}$ and $\ti {\textbf{m}}$ are equal on $\L$ for some $L, \ti L$.
Once again, without loss of generality we can assume that both operators are positive. Otherwise, we may add a large positive constant $a$ to both potentials, and  using  the transformation
$$
F(z)\mapsto F(\sqrt{z^2+a^2})
$$
for even  entire functions we observe  that $\sqrt\Lambda$  is a uniqueness set if $\sqrt{\Lambda+a}$ is.

Then, after the square root transform,
$m$ and $\ti m$ are equal on the set $\sqrt\L$. Also, by our definitions, $m(0)=\ti m(0)$. Hence $\ti \Theta = \Theta$ on $\sqrt\L\cup \{0\}$, i.e.
the function $(z-a)(\ti\Theta -\Theta)=0$ on $\sqrt\L\cup \{0,a\}$, 
where $a$ is any point not in $\sqrt\L\cup \{0\}$. By the definition of Weyl inner functions, the last equation translates into
$(z-a)(E^\# \ti E - \ti E^\# E)=0$ or equivalently $(z-a)(A \ti B - \ti A B)=0$. 
Since by statement 2 of Theorem \ref{char} the function $(z-a)(A \ti B - \ti A B)$ belongs to $PW_2$, we obtain a contradiction.
\medskip
\\
\textbf{Only if:} Without loss of generality, 
$0,1\not\in \sqrt\L$. Suppose that $\sqrt\L\cup\{0,1\}$ is not a uniqueness set for $PW_2$ and let $f\in PW_2$ be a non-trivial
function which vanishes on that set and real on $\mathbb{R}$. At least one of the functions $f(z)$ and $\frac{f(z)}{z-1}$ is not odd. Assume that $f$ is not odd. Put $\tilde{f}(z)=f(z)+f(-z)$. Clearly $\tilde{f}$ is a non-trivial even function.

By Theorem \ref{char}, $\tilde{f}=z(A\cos z - B\sin z)$ for some $E=A+iB$ corresponding to a Schroedinger operator $L$.
It is left to notice that then $\textbf{m}=\textbf{m}_0$ on $\L$, where $\textbf{m}$ corresponds to $L$ and $\textbf{m}_0$ corresponds to the free operator.
\end{proof}

\begin{remark}
{\rm In the second part of our proof we could obtain the following statement.
If $\sqrt\L\cup\{*,*\}$ is not a uniqueness set
in the Paley--Wiener class $PW_2$, then for any operator from $Schr(L^2,D)$ there exists 
another operator from the same class whose $\textbf{m}$-function takes the same values on $\L$.}
\end{remark}


\ms\section{Distribution of zeros of $E$}
\label{res}

In this section we study the distribution of zeros of an entire function $E$ corresponding to the Shroedinger
equation.

Recall that all zeros of a de Branges function belong to  $\mathbb{C}^-$. We will show that in any 
logarithmic strip there exists only finite number of zeros of $E$ corresponding to a Shroedinger equation
with $L^2$ potential.  

\begin{theorem}
Let $E$ be a de Branges function which corresponds to a Shroedinger equation
with $L^2$ potential.  Then there exists $C>0$ such that the logarithmic strip
$$
\Big \{z\in\mathbb{C^-}: -\frac{1}{2}\log(|\rea z|+2) + C \leq \ima z < 0 \Big\}
$$
contains only finite number of zeros of $E$.
\label{ZE}
\end{theorem}

We will need two lemmas.

\begin{lemma}
For any $\delta>0$ there exist two constants $c_\delta, C_\delta>0$ such that
$$c_\delta\leq\biggl{|}\frac{A(z)}{\sin \pi z}\biggr{|}\leq C_\delta,
\qquad c_\delta\leq\biggl{|}\frac{B(z)}{\cos \pi z}\biggr{|}\leq C_\delta,$$
 whenever $|\ima z|\geq\delta$. 
\label{l1}
\end{lemma}

\begin{proof}
Let $C_1$ be the constant defined by \eqref{sin2}. Recall the estimate \eqref{A1longest} from the proof of Corollary \ref{cor1}:

$$
\frac{A(z)}{C_1\sin z}  = \exp\biggl{[}\sum_{k\neq0}\frac{\pi k-\lambda_k}{z-\pi k}+o\biggl{(}\frac{1}{\dist(z,\pi\mathbb{Z})^2}\biggr{)}\biggr{]}
$$
$$
= \exp\biggl{[}-\sum_{k\neq0}\frac{c}{k(z-\pi k)}-\sum_{k\neq0}\frac{a_k}{z-\pi k}+o(1)\biggr{]},
$$
when $|z|\rightarrow\infty$, $|\ima z|\geq\delta$. Clearly, $\biggl{|}\sum_{k\neq0}\frac{a_k}{z-\pi k}\biggr{|}=O(1), |\ima z|\geq \delta$,
and 
$$\sum_{k\neq0}\frac{1}{k(z-\pi k)}=\frac{\pi}{z}\biggl{(}\cot z-\frac{1}{z}\biggr{)}.$$
Hence the expression in the exponent is bounded by a constant depending only on $\delta$,
but not on $z$.
\end{proof}

\begin{lemma}
There exists $\delta>0$ such that the strip $\{-\delta\leq\ima z<0\}$ 
contains no zeros of $E$.
\label{l2}
\end{lemma}

\begin{proof}
Let $\varphi$ be a phase function for $E$. It is well known that 
$$\varphi'(t)=\sum_n\frac{|\ima z_n|}{|t-z_n|^2}+a,$$
where $z_n$ are zeros of $E$ and $a$ is some non-negative real constant. In particular, $\varphi'(z_n)\geq \frac{1}{|\ima z_n|}$.
It is sufficient to show that for a function $E$ which corresponds to a Shroedinger equation we have $\varphi'\in L^\infty(\mathbb{R})$. 

Note that $\varphi'=|\Theta'|\slash2$, where $\Theta=\frac{m-i}{m+i}$ is the Weyl inner function.
Clearly, the modified Weyl function $m$ is given by
$$m(z)=\frac{A(z)}{B(z)}=\sum_n\frac{A(\mu_n)}{B'(\mu_n)}\biggl{(}\frac{1}{z-\mu_n}+\frac{1}{\mu_n}\biggr{)}.$$
Put $v_n=-\frac{A(\mu_n)}{B'(\mu_n)}$. By Corollary \ref{cor1}, we have $v_n=1+w_n$, $\{w_n\}\in\ell^2$. Then
$$\varphi'(t)=\frac{m'(t)}{|m(t)+i|^2}=\biggl{|}i+\sum_nv_n\biggl{(}\frac{1}{t-\mu_n}+\frac{1}{\mu_n}\biggr{)}\biggr{|}^{-2}
\times \sum_n\frac{v_n}{|t-\mu_n|^2}.$$
Clearly, for any $\varepsilon>0$ there exists $C_{\varepsilon}$ such that 
$$\varphi'(t)\leq \sum_n\frac{v_n}{|t-\mu_n|^2}\leq C_{\varepsilon}$$
whenever $\dist\{t,\{\mu_n\}\}\geq\varepsilon$.

On the other hand, if $\varepsilon$ is sufficiently small and $|t-\mu_k|<\varepsilon$ and $k$ is sufficiently large, we have
\begin{equation}
\biggl{|}\frac{1}{\mu_k}+\frac{1}{t-\mu_k}+\sum_{n\neq k}v_n\biggl{(}\frac{1}{t-\mu_n}+\frac{1}{\mu_n}\biggr{)}+
i\biggr{|}\asymp \frac{1}{|t-\mu_k|}
\label{thinest}
\end{equation}
and $\sum_n\frac{v_n}{|t-\mu_n|^2}\asymp \frac{1}{|t-\mu_k|^2}$ whence $\varphi'(t)\leq C_\varepsilon$ for some constant $C_\varepsilon$ independent of $t$.

In the estimate \eqref{thinest} we used that, by the asymptotics of $\mu_n$ and $v_n$,
$$
\biggl{|}\sum_{n\neq k}v_n\biggl{(}\frac{1}{t-\mu_n}+\frac{1}{\mu_n}\biggr{)}-\sum_{n}\biggl{(}\frac{1}{t-(\pi n+\pi\slash2)}+\frac{1}{\pi n +\pi\slash2}\biggr{)}\biggr{|}\lesssim1,
$$
when $|t-\mu_k|\leq\varepsilon$.
\end{proof}
\medskip
\noindent
{\it Proof of Theorem \ref{ZE}.} Let $\{z_n =x_n +iy_n\}$ 
denote the zeros of $E$. By Lemma~\ref{l2} there exists 
$\delta>0$ such that $y_n\leq -\delta$. By Theorem \ref{mainth} 
there exists $f\in PW_2$ and $C\in\mathbb{R}$ such that
$$
z(A(z)\cos z- B(z)\sin z)=f(z)+C.
$$
Since $E(z_n)=0$ we have $B(z_n)=iA(z_n)$.
Hence, $z_nA(z_n)e^{- iz_n} = f(z_n)+C$. By Lemma \ref{l1}, 
$|A(z_n)|\asymp |\sin z_n|$. Thus,
$$
|z_ne^{iz_n}|\lesssim \frac{|f(z_n)+C|}{|\sin z_n|} \lesssim e^{|y_n|},
$$ 
that is, $|x_n|+|y_n| \le  Me^{2|y_n|}$, for some $M>0$. Hence,
the logarithmic strip $|y|\leq \frac{1}{2} \log\bigl{|}\frac{x}{M}\bigr{|}$ 
contains no zeros of $E$ if $M$ is sufficiently big.
\qed

\begin{remark}
{\rm By Theorem \ref{mainth} the function
$$
E(z)=\frac{[z^2-9\pi^2\slash 16]\sin z}{z^2-\pi^2}+ 
i \cos z= \frac{(32z^2-25\pi^2)ie^{-iz}-7\pi^2 i e^{iz}}{32(z^2-\pi^2)}.
$$
corresponds to a Schroedinger operator with $L^2$ potential 
and all zeros of $E$ belong to a logarithmic strip
$\{-\log(|\rea z|)-M<\ima z<0\}$. 
Thus, the logarithmic form of the strip is optimal. However, in this example
the coefficient at $\log|\rea z|$ in the definition of the zero-free 
strip is 1, while in Theorem \ref{ZE} it is 1/2. }
\label{rem5}
\end{remark}


\begin{thebibliography}{99.}


\bibitem{abb} E. Abakumov, A. Baranov, Y. Belov, {\it Localization of zeros for Cauchy transforms,} 
Int. Math. Res. Notices, 2015, (2015), 6699--6733 

\bibitem{dB} L.~de Branges, \textit{Hilbert spaces of entire functions}, 
Prentice--Hall, Englewood Cliffs, NJ, 1968.

\bibitem{Chelkak} D. Chelkak, \textit{An application of the fixed point theorem
to the inverse Sturm--Liouville problem,} 
Zap. Nauchn. Sem. S.-Peterburg. Otdel. Mat.
Inst. Steklov. (POMI) 370 (2009);  English
transl. in J. Math. Sci. 166 (2010), 1, 118--126.

\bibitem{DM} H.~Dym, H.P.~McKean, \textit{Gaussian processes, 
function theory and the inverse spectral problem}, Academic Press, New York, 1976.

\bibitem{Horvath} M. Horv\'ath, {\it Inverse spectral 
problems and closed exponential systems,} Ann. Math.
162 (2005), 2, 885--918.

\bibitem {hnp} S.V. Hruscev, N.K. Nikolskii, B.S. Pavlov,
\textit{Unconditional bases of exponentials and of reproducing kernels},
Lecture Notes in Math.  864 (1981), 214--335.

\bibitem{korot} A. Iantchenko, E. Korotyaev, 
{\it Resonances for Dirac operators on the half-line,}
J. Math. Anal. Appl. 420 (2014), no. 1, 279--313.

\bibitem{Lagarias1} J. Lagarias, {\it Zero spacing distributions for 
differenced L-functions,}  Acta Arithmetica 120 (2005), 2, 159--184.

\bibitem{Lagarias2} J. Lagarias, {\it The Schr\"odinger operator with Morse 
potential on the right half line}, Comm. Number Theory Phys. 3 (2009), 2, 323--361.

\bibitem{march1} V.A. Marchenko, 
{\it Certain problems in the theory of second-order differential operators}, 
Dokl. Akad. Nauk SSSR 72 (1950), 457--460 (Russian).

\bibitem{march2} V.A. Marchenko, 
{\it Some questions in the theory of one-dimensional linear differential operators of the second
order.} I, Trudy Moskov. Mat. Ob. 1 (1952), 327--420; 
English transl. in Amer. Math. Soc. Transl. (2) 101 (1973), 1--104.

\bibitem{MIF} N. Makarov,  A. Poltoratski, 
{\it Meromorphic inner functions, Toeplitz kernels, and the uncertainty principle,} in {\it Perspectives in Analysis}, 
Carleson's Special Volume, Springer Verlag, Berlin, 2005, 185--252.

\bibitem{Romanov2} R. Romanov, {\it Canonical systems and de Branges spaces,} 
http://arxiv.org/abs/1408.6022.

\bibitem{Romanov1} R. Romanov, {\it Order problem for canonical systems and
a conjecture of Valent,} http://arxiv.org/abs/1502.04402, to appear
in Trans. Amer. Math. Soc.

\bibitem {seip} 
K. Seip, {\it Interpolation and Sampling in 
Spaces of Analytic Functions}, 
Univ. Lect. Ser., Vol. 33, AMS, Providence, RI, 2004.

\bibitem{simon} B. Simon, {\it Resonances in one dimension and Fredholm determinants,} 
J. Funct. Anal. 178 (2) (2000) 396--420.

\end{thebibliography}
\end{document}